\documentclass{article}

\usepackage[final]{nips_2016}

\usepackage{graphicx}
\usepackage{subcaption}

\usepackage{pbox}

\usepackage{enumitem}

\usepackage[utf8]{inputenc} 
\usepackage[T1]{fontenc}    
\usepackage{hyperref}       
\usepackage{url}            
\usepackage{booktabs}       
\usepackage{amsfonts}       
\usepackage{nicefrac}       
\usepackage{microtype}      
\usepackage{amsmath, amsfonts, amsthm, amssymb}

\newtheorem{theorem}{Theorem}
\newtheorem{lemma}[theorem]{Lemma}

\newtheorem{corollary}[theorem]{Corollary}
\newtheorem{definition}[theorem]{Definition}
\newtheorem{remark}[theorem]{Remark}
\renewenvironment{proof}{{\bf Proof:}}{\hfill\rule{2mm}{2mm}}

\newcommand{\inv}{^{-1}}                            
\newcommand{\sminus}{\backslash}                    
\newcommand{\R}{\mathbb{R}}                         
\newcommand{\e}{\varepsilon}                        
\newcommand{\X}{\mathcal{X}}                        
\newcommand{\XX}{\mathbb{X}}                        
\newcommand{\B}{\mathcal{B}}                        
\newcommand{\E}{\mathop{\mathbb{E}}}                
\newcommand{\Var}{\mathbb{V}}                       
\newcommand{\pr}{\mathbb{P}}                        
\newcommand{\dist}{\operatorname{dist}}             
\newcommand{\diam}{\operatorname{diam}}             
\newcommand{\acro}[1]{\textsc{\MakeLowercase{#1}}}
\renewcommand{\phi}{\varphi} 

\title{Finite-Sample Analysis of Fixed-$k$ Nearest Neighbor Density Functional
Estimators}

\author{
  Shashank Singh \\
  Statistics \& Machine Learning Departments \\
  Carnegie Mellon University \\
  Pittsburgh, PA 15213 \\
  \texttt{sss1@andrew.cmu.edu} \\
  \And
  Barnab\'as P\'oczos \\
  Machine Learning Departments \\
  Carnegie Mellon University \\
  Pittsburgh, PA 15213 \\
  \texttt{bapoczos@cs.cmu.edu} \\
}

\begin{document}

\maketitle

\begin{abstract}
We provide finite-sample analysis of a general framework for using $k$-nearest
neighbor statistics to estimate functionals of a nonparametric continuous
probability density, including entropies and divergences. Rather than plugging
a consistent density estimate (which requires $k \to \infty$ as the sample size
$n \to \infty$) into the functional of
interest, the estimators we consider fix $k$ and perform a bias correction.
This is more efficient computationally, and, as we show in certain cases,
statistically, leading to faster convergence rates. Our framework unifies
several previous estimators, for most of which ours are the first finite sample
guarantees.
\end{abstract}

\section{Introduction}
\label{sec:introduction}
Estimating entropies and divergences of probability distributions in a
consistent manner is of importance in a number problems in machine learning.
Entropy estimators have applications in
goodness-of-fit testing \citep{goria05new},
parameter estimation in semi-parametric models \citep{Wolsztynski85minimum},
studying fractal random walks \citep{Alemany94fractal},
and texture classification \citep{hero02alpha,hero2002aes}.
Divergence estimators have been used to generalize machine learning algorithms
for regression, classification, and clustering from inputs in $\R^D$ to sets and
distributions \citep{poczos12kernelimages,oliva13ICML}.

Divergences also include mutual informations as a special case; mutual
information estimators have applications in
feature selection \citep{peng05feature},
clustering \citep{aghagolzadeh07hierarchical},
causality detection \citep{Hlavackova07causality},
optimal experimental design \citep{lewi07realtime, poczos09identification},
f\acro{MRI} data analysis \citep{chai09exploring},
prediction of protein structures \citep{adami04information},
and boosting and facial expression recognition \cite{Shan05conditionalmutual}.
Both entropy estimators and mutual information estimators have been used for
independent component and subspace analysis
\citep{radical03,szabo07undercomplete_TCC, poczos05geodesic,Hulle08constrained},
as well as for image registration
\citep{kybic06incremental,hero02alpha,hero2002aes}.
Further applications can be found in \citet{Leonenko-Pronzato-Savani2008}.

This paper considers the more general problem of using $n$ IID samples from $P$
to estimate functionals of the form
\begin{equation}
F(P) := \E_{X \sim P} \left[ f(p(X)) \right],
\label{eq:func_form}
\end{equation}
where $P$ is an unknown probability measure with smooth density function $p$
and $f$ is a known smooth function. We are interested in analyzing a class of
nonparametric estimators based on $k$-nearest neighbor ($k$-NN) distance
statistics. Rather than plugging a consistent estimator of $p$ into
(\ref{eq:func_form}), which requires $k \to \infty$ as $n \to \infty$, these
estimators derive a bias correction for the plug-in estimator with \emph{fixed}
$k$; hence, we refer to this type of estimator as a fixed-$k$ estimator.
Compared to plug-in estimators, fixed-$k$ estimators are faster to compute. As
we show, fixed-$k$ estimators can also exhibit superior rates of convergence.

As shown in Table \ref{table:bias_corrections}, several authors have derived
bias corrections necessary for fixed-$k$ estimators of entropies and
divergences, including, most famously, the Shannon entropy estimator of
\citet{kozachenko87statistical}.
\footnote{MATLAB implementations of many of these estimators can be found in
the Information Theoretical Estimators toolbox available at
\url{https://bitbucket.org/szzoli/ite/}. \citep{szabo14ITE}.}
The estimators in Table
\ref{table:bias_corrections} estimators are known to be weakly consistent.
\footnote{Several of these proofs contain errors regarding the use of integral
convergence theorems when their conditions do not hold, as described in
\citet{poczos11AISTATS}.}
However, for most of these estimators, no finite sample bounds are known. The
{\bf main goal of this paper} is to provide finite-sample analysis of these
estimators, via a unified analysis of the estimator after bias correction.
Specifically, we will show conditions under which, for $\beta$-H\"older
continuous ($\beta \in (0, 2]$) densities on $D$ dimensional space, the bias of
fixed-$k$ estimators decays as $O \left( n^{-\beta/D} \right)$ and
the variance decays as $O \left( n \inv \right)$, giving a mean squared error
of $O \left( n^{-2\beta/D} + n\inv \right)$. Hence, the estimators
converge at the parametric $O(n\inv)$ rate when $\beta \geq D/2$, and at the
slower rate $O(n^{-2\beta/D})$ otherwise. A modification of the estimators
would be necessary to leverage additional smoothness for $\beta > 2$, but we do
not pursue this here. Along the way, we also prove a finite-sample version of
the useful fact \citep{Leonenko-Pronzato-Savani2008} that (appropriately
normalized) $k$-NN distances have an asymptotic Erlang distribution, which may
be of independent interest.

\begin{table}
{\renewcommand{\arraystretch}{1.1} 
\begin{tabular}{|p{24mm}|c|p{34mm}|p{40mm}|}
\hline
Functional Name & Functional Form & Correction & Reference \\
\hline
Shannon Entropy & $\E \left[ \log p(X) \right]$ & Additive constant: $\psi(n) -
\psi(k) + \log(k/n)$ & \citet{kozachenko87statistical}\citet{goria05new} \\
\hline
R\'enyi-$\alpha$ Entropy & $\E \left[ p^{\alpha - 1}(X) \right]$ &
Multiplicative constant: $\frac{\Gamma(k)}{\Gamma(k + 1 - \alpha)}$ & \citet{Leonenko-Pronzato-Savani2008,leonenko10correction} \\
\hline
KL Divergence & $\E \left[ \log \frac{p(X)}{q(X)} \right]$ & None$^*$ & \citet{Wang-Kulkarni-Verdu2009} \\
\hline
$\alpha$-Divergence & $\E \left[ \left( \frac{p(X)}{q(X)} \right)^{\alpha - 1} \right]$ & Multiplicative constant: $\frac{\Gamma^2(k)}{\Gamma(k - \alpha + 1)\Gamma(k + \alpha - 1)}$ & \citet{poczos11AISTATS} \\
\hline
\end{tabular}
}
\caption{Table of functionals with known bias-corrected $k$-NN estimators, the
type of bias correction necessary, the correction constant, and references. All
expectations are over $X \sim P$. $\Gamma(t) = \int_0^\infty x^{t - 1}
e^{-x} \, dx$ is the gamma function, and
$\psi(x) = \frac{d}{dx} \log \left( \Gamma(x) \right)$ is the digamma function.
$\alpha$ is a parameter in $\R \sminus \{1\}$.
$^*$For KL divergence, the bias corrections for $p$ and $q$ exactly cancel.
}
\label{table:bias_corrections}
\end{table}

We present our results for distributions $P$ supported on the unit cube in
$\R^D$ because this significantly simplifies the statements of our results,
but, as we discuss in the supplement, our results generalize fairly naturally,
for example to to distributions supported on a smooth compact manifold. In this
context, it is worth noting that our results would scale with the
\emph{intrinsic} dimension of the manifold. As we discuss later, we believe
that deriving finite sample rates for distributions with \emph{unbounded}
support may require a truncated modification of the estimators we study (as in
\citet{tsybakov96rootn}), but we do not pursue this modification here.

\section{Problem statement and notation}
\label{sec:problem}
Let $\X := [0, 1]^D$ denote the unit cube in $\R^D$, and let $\mu$ denote the
Lebesgue measure. 
Suppose $P$ is an unknown $\mu$-absolutely continuous Borel probability
measure supported on $\X$, and let $p : \X \to [0, \infty)$ denote the density
of $P$. Consider a (known) differentiable function $f : (0, \infty) \to \R$.
Given $n$ samples $X_1,...,X_n$ drawn IID from $P$, we are interested in
estimating the functional
\[F(P) := \E_{X \sim P} \left[ f(p(X)) \right].\]

Somewhat more generally (as in divergence estimation), we may have a function
$f : (0, \infty)^2 \to \R$ of two variables and a second unknown
probability measure $Q$, with density $q$ and $n$ IID samples $Y_1,...,Y_n$.
Then, we are interested in estimating 
\[F(P, Q) := \E_{X \sim P} \left[ f(p(X), q(X)) \right].\]

Fix $r \in [1, \infty]$ and a positive integer $k$. We will work with distances
induced by the $r$-norm
\[\|x\|_r := \left( \sum_{i = 1}^D x_i^r \right)^{1/r}
  \quad \mbox{ and define } \quad
  c_{D,r}
  := \frac{\left( 2\Gamma(1 + 1/r) \right)^D}{\Gamma(1 + D/r)}
  = \mu(B(0, 1)),\]
where $B(x, \e) := \{y \in \R^D : \|x - y\|_r < \e\}$ denotes the open
radius-$\e$ ball centered at $x$. Our estimators use $k$-nearest neighbor
($k$-NN) distances:
\begin{definition}
{\bf ($k$-NN distance):}
Suppose we have $n$ samples $X_1,...,X_n$ drawn IID from $P$. For any
$x \in \R^D$, we define the \emph{$k$-nearest neighbor distance} $\e_k(x)$ by
$\e_k(x) = \|x - X_i\|_r$, where $X_i$ is the $k^{th}$-nearest element (in
$\|\cdot\|_r$) of the set $\{X_1,...,X_n\}$ to $x$. For divergence estimation,
if we also have $n$ samples $Y_1,...,Y_n$ drawn IID from $Q$, then we similarly
define $\delta_k(x)$ by $\delta_k(x) = \|x - Y_i\|_r$, where $Y_i$ is the
$k^{th}$-nearest element of $\{Y_1,...,Y_n\}$ to $x$.
\label{def:kNN_dist}
\end{definition}

Note that the $\mu$-absolute continuity of $P$ precludes the existence of atoms
(i.e., for all $x \in \R^D$, $P(\{x\}) = \mu(\{x\}) = 0$). Hence, for all
$x \in \R^D$, $\e_k(x) > 0$ almost surely. This is important, since we will
consider quantities such as $\log \e_k(x)$ and $\frac{1}{\e_k(x)}$.

\section{Estimator}
\label{sec:estimator}

\subsection{$k$-NN density estimation and plug-in functional estimators}
The $k$-NN density estimator
\[\hat p_k(x)
  = \frac{k/n}{\mu(B(x, \e_k(x))}
  = \frac{k/n}{c_D \e_k^D(x)}\]
is well-studied nonparametric density estimator (originally due to
\citet{Loftsgaarden65nonparametric}), motivated by the observations that, for
small $\e > 0$,
\[p(x) \approx \frac{P(B(x, \e))}{\mu(B(x, \e))},\]
and that, $P(B(x, \e_k(x))) \approx k/n$. One can show that, for $x \in \R^D$
at which $p$ is continuous, if $k \to \infty$ and $k/n \to 0$ as
$n \to \infty$, then $\hat p_k(x) \to p(x)$ in probability
(\citet{Loftsgaarden65nonparametric}, Theorem 3.1). Thus, a natural approach
for estimating $F(P)$ is the plug-in estimator
\begin{equation}
\hat F_{PI}
  := \frac{1}{n}
     \sum_{i = 1}^n f \left( \hat p_k(X_i) \right).
\label{eq:plug_in_est}
\end{equation}
Since $\hat p_k \to p$ in probability pointwise as $k, n \to \infty$ and $f$ is
smooth, one can show $\hat F_{PI}$ is consistent, and in fact derive
finite sample convergence rates (depending on how $k \to \infty$).
For example, \citet{sricharan10confidence} show a convergence rate of
$O \left( n^{-\min \left\{ \frac{2\beta}{\beta + D}, 1 \right\}} \right)$ for
$\beta$-H\"older continuous densities (after sample splitting and boundary
correction) by setting $k \asymp n^{\frac{\beta}{\beta + d}}$.

Unfortunately, while necessary to ensure
$\Var \left[ \hat p_k(x) \right] \to 0$, the requirement $k \to \infty$ is
computationally burdensome. Furthermore, increasing $k$ can increase the bias
of $\hat p_k$ due to over-smoothing (see \ref{ineq:KNN_density_est_bias}
below), suggesting that this may be sub-optimal for estimating $F(P)$. Indeed,
similar work based on kernel density estimation \citep{singh14densityfuncs}
suggests that, for plug-in functional estimators, \emph{under-smoothing} may be
preferable, since the empirical mean results in additional smoothing.

\subsection{Fixed-$k$ functional estimators}

An alternative approach is to fix $k$ as $n \to \infty$. Since
$\hat F_{PI}$ is itself an empirical mean, unlike
$\Var \left[ \hat p_k(x) \right]$, $\Var \left[ \hat F_{PI} \right] \to 0$ as
$n \to \infty$.

A more critical complication of fixing $k$ is bias. Since $f$ is typically
non-linear, the non-vanishing variance of $\hat p_k$ translates into asymptotic
bias. A solution adopted by several papers is to derive a bias correction
function $\B$ (depending only on known factors) such that
\begin{equation}
\E_{X_1,...,X_n} \left[
      \B \left(
          f \left(
              \frac{k/n}{\mu(B(x, \e_k(x))}
          \right)
      \right)
  \right]
  = \E_{X_1,...,X_n} \left[
        f \left(
            \frac{P(B(x, \e_k(x)))}{\mu(B(x, \e_k(x))}
        \right)
    \right].
\label{eq:bias_correction}
\end{equation}
For continuous $p$, the quantity
\begin{equation}
p_{\e_k(x)}(x) := \frac{P(B(x, \e_k(x)))}{\mu(B(x, \e_k(x))}
\label{eq:adaptive_KDE}
\end{equation}
\emph{is} a consistent estimate of $p(x)$ with $k$ fixed, but it is not
computable, since $P$ is unknown. The bias correction $\B$ gives us an
asymptotically unbiased estimator
\[\hat F_\B(P)
  := \frac{1}{n} \sum_{i = 1}^n
      \B \left(
          f \left(
            \hat p_k(X_i)
          \right)
      \right)
  = \frac{1}{n} \sum_{i = 1}^n
      \B \left(
          f \left(
              \frac{k/n}{\mu(B(X_i, \e_k(X_i))}
          \right)
      \right)
.\]
that uses $k/n$ in place of $P(B(x, \e_k(x)))$. This estimate extends naturally
to divergences:
\[\hat F_\B(P, Q)
  := \frac{1}{n} \sum_{i = 1}^n
      \B \left(
          f \left(
            \hat p_k(X_i),
            \hat q_k(X_i)
          \right)
      \right).\]

As an example, if $f = \log$ (as in Shannon entropy), then it can be shown
that, for any continuous $p$,
\[\E \left[ \log P(B(x, \e_k(x))) \right] = \psi(k) - \psi(n).\]
Hence, for $B_{n,k} := \psi(k) - \psi(n) + \log(n) - \log(k)$,
\[\E_{X_1,...,X_n} \left[
      f \left(
          \frac{k/n}{\mu(B(x, \e_k(x))}
      \right)
  \right]
  + B_{n,k}
  = \E_{X_1,...,X_n} \left[
        f \left(
            \frac{P(B(x, \e_k(x)))}{\mu(B(x, \e_k(x))}
        \right)
    \right].\]
giving the estimator of \citet{kozachenko87statistical}.
Other examples of functionals for which the bias correction is known are given
in Table \ref{table:bias_corrections}.

In general, deriving an appropriate bias correction can be quite a difficult
problem specific to the functional of interest, and it is not our goal
presently to study this problem; rather, we are interested in bounding the
error of $\hat F_\B(P)$, \emph{assuming the bias correction is known.} Hence,
our results apply to all of the estimators in Table
\ref{table:bias_corrections}, as well as any estimators of this form that may
be derived in the future.

\section{Related work}
\label{sec:related_work}

\subsection{Estimating information theoretic functionals}
Quite recently, there has been much work on analyzing new estimators for
entropy, mutual information, divergences, and other functionals of densities.
Besides bias-corrected fixed-$k$ estimators, most of this work has been along
one of three approaches. One series of papers
\citep{liu12exponential,singh14divergence,singh14densityfuncs} studied a
boundary-corrected plug-in approach based on under-smoothed kernel density
estimation. This approach has strong finite sample guarantees, but requires
prior knowledge of the support of the density and can necessitate
computationally demanding numerical integration. A second approach
\citep{krishnamurthy14divergences,kandasamy15vonMises} uses von Mises expansion
to correct the bias of optimally smoothed density estimates. This approach
shares the difficulties of the previous approach, but is statistically more
efficient. A final line of work
\citep{sricharan10confidence,sricharan12ensemble,moon14ensemble,moon14confidence}
has studied entropy estimation based on plugging in consistent, boundary
corrected $k$-NN density estimates (i.e., with $k \to \infty$ as
$n \to \infty$). There is also a divergence estimator \citep{nguyen10estimating}
based on convex risk minimization, but this is framed in the context of an RKHS
and results are difficult to compare.

{\bf Rates of Convergence:} For densities over $\R^D$ satisfying a H\"older
smoothness condition parametrized by $\beta \in (0, \infty)$, the minimax mean
squared error rate for estimating functionals of the form $\int f(p(x)) \, dx$
has been known since \citet{birge95estimation} to be
$O \left( n^{-\min \left\{ \frac{8\beta}{4\beta + D}, 1 \right\}} \right)$.
\citet{krishnamurthy14divergences} recently derived identical minimax rates for
divergence estimation.

Most of the above estimators have been shown to converge at the rate
$O \left( n^{-\min \left\{ \frac{2\beta}{\beta + D}, 1 \right\}} \right)$.
Only the von Mises approach of \citet{krishnamurthy14divergences} is
known to achieve the minimax rate for general $\beta$ and $D$, but due to its
high computational demand ($O(2^D n^3)$), the authors suggest the use
of other statistically less efficient estimators for moderately sized datasets.
In this paper, we show that, for $\beta \in (0, 2]$, bias-corrected fixed-$k$
estimators converge at the relatively fast rate of
$O \left( n^{-\min \left\{ \frac{2\beta}{D}, 1 \right\}} \right)$. For
$\beta > 2$, modifications are needed for the estimator to leverage the
additional smoothness of the density. It is also worth noting the relative
computational efficiency of the fixed-$k$ estimators
($O \left( D n^2 \right)$, or $O \left( 2^D n \log n \right)$ using
$k$-d trees for small $D$).

\subsection{Prior analysis of fixed-$k$ estimators}

To our knowledge, the only finite-sample results for $\hat F_\B(P)$ are the
recent results of \citet{biau15EntropyKNN} for the Kozachenko-Leonenko (KL)
\footnote{Not to be confused with Kullback-Leibler (KL) divergence, for which
we also analyze an estimator.} Shannon entropy estimator.
\citep{kozachenko87statistical} Theorem 7.1 of
\citet{biau15EntropyKNN} shows that, if the density $p$ has compact support,
then the variance of the KL estimator decays as $O(n\inv)$. They also claim
(Theorem 7.2) to bound the bias of the KL estimator by $O(n^{-\beta})$, under
the assumptions that $p$ is $\beta$-H\"older continuous ($\beta \in (0, 1]$),
bounded away from $0$, and supported on the interval $[0, 1]$. However, in
their proof \citet{biau15EntropyKNN} neglect to bound the additional bias
incurred near the boundaries of $[0, 1]$, where the density cannot
simultaneously be bounded away from $0$ and continuous. In fact, because the
KL estimator does not attempt to correct for boundary bias, it is not clear
that the bias should decay as $O(n^{-\beta})$ under these conditions; we will
require additional conditions at the boundary of $\X$.

\citet{tsybakov96rootn} studied a closely related entropy estimator for which
they prove $\sqrt{n}$-consistency. Their estimator is identical to the KL
estimator, except that it truncates $k$-NN distances at $\sqrt{n}$, replacing
$\e_k(x)$ with $\min\{ \e_k(x), \sqrt{n} \}$. This sort of truncation may be
necessary for certain fixed-$k$ estimators to satisfy finite-sample bounds for
densities of \emph{unbounded} support, although consistency can be shown
regardless.

\section{Discussion of assumptions}
The lack of finite-sample results for fixed-$k$ estimators is due to several
technical challenges. Here, we discuss some of these challenges, motivating the
assumptions we make to overcome them.

First, these estimators are sensitive to regions of low probability (i.e.,
$p(x)$ small), for two reasons:
\begin{enumerate}
\item
Many functions $f$ of interest (e.g., $f = \log$ or $f(z) = z^\alpha$,
$\alpha < 0$) have singularities at $0$.
\item
The $k$-NN estimate $\hat p_k(x)$ of $p(x)$ is highly biased when $p(x)$ is
small. For example, for $p$ $\beta$-H\"older continuous ($\beta \in (0, 2]$),
one has (\citep{mack79KNNDensityEstimate}, Theorem 2)
\begin{equation}
\mbox{Bias}(\hat p_k(x)) \asymp \left( \frac{k}{n p(x)} \right)^{\beta/D}.
\label{ineq:KNN_density_est_bias}
\end{equation}
\end{enumerate}
For these reasons, it has been common in the analysis of $k$-NN estimators to
make the following assumption: \citep{poczos11AISTATS,biau15EntropyKNN}
\begin{enumerate}[label=\textbf{(A\arabic*)}]
\item
$p$ is bounded away from zero on its support. That is,
$p_* := \inf_{x \in \X} p(x) > 0$.
\label{item:A1}
\end{enumerate}
Second, unlike many functional estimators (see e.g.,
\citet{pal10estimation,sricharan12confidence,singh14densityfuncs}), the
fixed-$k$ estimators we consider do not attempt correct for boundary bias (i.e.,
bias incurred due to discontinuity of $p$ on the boundary $\partial\X$ of
$\X$).
\footnote{This complication appears to have been omitted in the bias bound
(Theorem 7.2) of \citet{biau15EntropyKNN} for entropy estimation.}
The boundary bias of the density estimate $\hat p_k(x)$ does vanish at $x$ in
the interior $\X^\circ$ of $\X$ as $n \to \infty$, but additional assumptions
are needed to obtain finite-sample rates. Either of the following assumptions
would suffice:
\begin{enumerate}[label=\textbf{(A\arabic*)}]
\setcounter{enumi}{1}
\item
$p$ is continuous not only on $\X^\circ$ but also on $\partial\X$ (i.e.,
$p(x) \to 0$ as $\dist(x, \partial \X) \to 0$).
\label{item:A2}
\item
$p$ is supported on all of $\R^D$. That is, the support of $p$ has no boundary.
This is the approach of \citet{tsybakov96rootn}, but we reiterate that, to
handle an unbounded domain, they require truncating $\e_k(x)$.
\label{item:A3}
\end{enumerate}
Unfortunately, both assumptions \ref{item:A2} and \ref{item:A3} are
inconsistent with \ref{item:A1}. Our approach is to assume \ref{item:A2} and
replace assumption \ref{item:A1} with a much milder assumption that $p$ is
\emph{locally lower bounded} on its support in the following sense:
\begin{enumerate}[label=\textbf{(A\arabic*)}]
\setcounter{enumi}{3}
\item
There exist $\rho > 0$ and a function $p_* : \X \to (0, \infty)$ such that, for
all $x \in \X, r \in (0, \rho]$,
$p_*(x) \leq \frac{P(B(x, r))}{\mu(B(x, r))}$.
\label{item:A4}
\end{enumerate}
We will show (Lemma \ref{lemma:p_*_exists}) that assumption \ref{item:A4} is in
fact very mild; in a metric measure space of positive dimension $D$, as long as
$p$ is continuous on $\X$, such a $p_*$ exists for \emph{any} desired
$\rho > 0$. For simplicity, we will use $\rho = \sqrt{D} = \diam(\X)$.

As hinted by \eqref{ineq:KNN_density_est_bias} and the fact that $F(P)$ is an
expectation, our bounds will contain terms of the form
\[\E_{X \sim P} \left[ \frac{1}{\left( p_*(X) \right)^{\beta/D}} \right]
  = \int_\X \frac{p(x)}{\left( p_*(x) \right)^{\beta/D}} \, d\mu(x)\]
(with an additional $f'(p_*(x))$ factor if $f$ has a singularity at zero).
Hence, the real non-trivial assumptions we make will be that these quantities
are finite. This depends primarily on how quickly $p$ can be allowed to
approach zero near $\partial \X$ (which may be $\infty$ if $\X$ is
unbounded). For many functionals, Lemma \ref{lemma:negative_integral_bound}
will give a simple sufficient condition.

\section{Preliminary lemmas}
Here, we present some lemmas, both as a means of summarizing our proof
techniques and also because they may be of independent interest for proving
finite-sample bounds for other $k$-NN methods. Due to space constraints, all
proofs are given in the appendix. Our first lemma states that, if $p$ is
continuous, then it is locally lower bounded as described in the previous
section.

\begin{lemma}
{\bf (Existence of Local Bounds)}
If $p$ is continuous on $\X$ and strictly positive on the interior $\X^\circ$
of $\X$, then, for $\rho := \sqrt{D} = \diam(\X)$, there exists a continuous
function $p_* : \X^\circ \to (0, \infty)$ and a constant $p^* \in (0, \infty)$
such that
\[0
  < p_*(x)
  \leq \frac{P(B(x, r))}{\mu(B(x, r))}
  \leq p^*
  < \infty,
  \quad \forall x \in \X, r \in (0, \rho].\]
\label{lemma:p_*_exists}
\end{lemma}

We now show that the existence of local lower and upper bounds implies
concentration of the $k$-NN distance of around a term of order
$\left( \frac{k}{n p(x)} \right)^{1/D}$. Related lemmas, also based on
multiplicative Chernoff bounds, have been used by
\citet{kpotufe11KNNClusterTrees,chaudhuri14clusterTreePruning} and
\citet{chaudhuri14KNNClassification,kontorovich15Consistent1NN}
to prove finite-sample bounds on $k$-NN methods for cluster tree pruning and
classification, respectively. For cluster tree pruning, the relevant
inequalities bound the error of the $k$-NN density estimate, and, for
classification, they lower bound the probability of nearby samples of the same
class. Unlike in cluster tree pruning, we are not using a consistent density
estimate, and, unlike in classification, our estimator is a function of $k$-NN
distances themselves (rather than their ordering). Hence, our statement is
somewhat different, bounding the $k$-NN distances themselves:
\begin{lemma}
{\bf (Concentration of $k$-NN Distances)}
Suppose $p$ is continuous on $\X$ and strictly positive on $\X^\circ$. Let
$p_*$ and $p^*$ be as in Lemma \ref{lemma:p_*_exists}. Then, for any
$x \in \X^\circ$,
\begin{enumerate}
\item 
if $r > \left( \frac{k}{p_*(x) n} \right)^{1/D}$,
\quad then \quad
$\displaystyle \pr \left[ \e_k(x) > r \right]
  \leq e^{-p_*(x) r^D n} \left( e\frac{p_*(x) r^D n}{k} \right)^k$.
\item
if $r \in \left[ 0, \left( \frac{k}{p^* n} \right)^{1/D}\right)$,
\quad then \quad
$\displaystyle \pr \left[ \e_k(x) < r \right]
  \leq e^{-p_*(x) r^D n} \left( \frac{e p^* r^D n}{k} \right)^{kp_*(x)/p^*}$.
\end{enumerate}
\label{lemma:KNN_concentration}
\end{lemma}
It is worth noting the asymmetry of the upper and lower bounds; perhaps
counter-intuitively, the lower bound also depends on $p_*$. It is this
asymmetry that causes the large (over-estimation) bias of $k$-NN density
estimators when $p$ is small (as in \eqref{ineq:KNN_density_est_bias}).

The following theorem uses Lemma \ref{lemma:KNN_concentration} to bound
expectations of monotone functions of $\hat p_k$ normalized by $p_*$. As
suggested by the form of the integral in the bounds, this can be thought of as a
finite-sample statement of the fact that (appropriately normalized) $k$-NN
distances have an asymptotic Erlang distribution; this asymptotic statement is
central to the consistency proofs of \citet{Leonenko-Pronzato-Savani2008} and
\citet{poczos11AISTATS} for their $\alpha$-entropy and divergence estimators,
respectively.

\begin{lemma}
Suppose $p$ is continuous on $\X$ and strictly positive on $\X^\circ$. Let
$p_*$ and $p^*$ be as in Lemma \ref{lemma:p_*_exists}. Suppose
$f : (0, \infty) \to \R$ is continuously differentiable, with $f' > 0$. Then,
for any $x \in \X^\circ$, we have the upper bound
\footnote{$f_+(x) = \max\{0, f(x)\}$ and $f_-(x) = -\min\{0, f(x)\}$ denote
the positive and negative parts of $f$. Recall that
$\E \left[ f(X) \right] = \E \left[ f_+(X) \right] - \E \left[ f_-(X) \right]$.
}
\begin{align}
\E \left[ f_+ \left( \frac{p_*(x)}{\hat p_k(x)} \right) \right]
& \leq f_+(1)
  + e\sqrt{k}
    \int_k^\infty 
        \frac{e^{-y} y^k}{\Gamma(k + 1)} f_+ \left( \frac{y}{k} \right)
    \, dy,
\label{ineq:KNN_functional_upper_bdd}
\end{align}
and, for $\kappa(x) := k p_*(x)/p^*$, the lower bound
\begin{align}
\E \left[ f_- \left( \frac{p_*(x)}{\hat p_k(x)} \right) \right]
& \leq f_-(1)
  + e \sqrt{\frac{k}{\kappa(x)}}
    \int_0^{\kappa(x)} \frac{e^{-y} y^{\kappa(x)}}{\Gamma(\kappa(x) + 1)} f_- \left( \frac{y}{k} \right) \, dy
\label{ineq:KNN_functional_lower_bdd}
\end{align}
\label{lemma:KNN_functional_bdd}
\end{lemma}

Note that plugging the function
$z \mapsto f \left( \left( \frac{kz}{c_{D,r} n p_*(x)} \right)^{\frac{1}{D}} \right)$
into Lemma \ref{lemma:KNN_functional_bdd} gives bounds on
$\E \left[ f(\e_k(x)) \right]$. As one might guess from Lemma
\ref{lemma:KNN_concentration} and the assumption that $f$ is smooth, this bound
is roughly of the order $\asymp \left( \frac{k}{n p(x)} \right)^\frac{1}{D}$.
For example, for any $\alpha > 0$, a simple calculation from
(\ref{ineq:KNN_functional_upper_bdd}) gives
\begin{equation}
\E \left[ \e_k^\alpha(x) \right]
 \leq \left( 1 + \frac{\alpha}{D} \right)
       \left( \frac{k}{c_{D,r} n p_*(x)} \right)^{\frac{\alpha}{D}}.
\label{ineq:pos_moment_stat}
\end{equation}
\eqref{ineq:pos_moment_stat} is used for our bias bound, and more direct
applications of Lemma \ref{lemma:KNN_functional_bdd} are used in variance
bound.

\section{Main results}
\label{sec:results}

Here, we present our main results on the bias and variance of $\hat F_\B(P)$.
Again, due to space constraints, all proofs are given in the appendix. We begin
with bounding the bias:

\begin{theorem}
{\bf (Bias Bound)}
Suppose that, for some $\beta \in (0, 2]$, $p$ is $\beta$-H\"older continuous
with constant $L > 0$ on $\X$, and $p$ is strictly positive on $\X^\circ$.
Let $p_*$ and $p^*$ be as in Lemma \ref{lemma:p_*_exists}. Let
$f : (0, \infty) \to \R$ be differentiable, and define
$M_{f,p} : \X \to [0, \infty)$ by
\[M_{f,p}(x)
  := \sup_{z \in \left[ p_*(x), p^* \right]} \left| \frac{d}{dz} f(z) \right|\]
Assume
\[C_f
  := \E_{X \sim p} \left[
        \frac{M_{f,p}(X)}{\left( p_*(X) \right)^{\frac{\beta}{D}}}
  \right]
  < \infty.
\quad \mbox{ Then, } \quad
\left| \hat F_\B(P) - F(P) \right|
  \leq C_f L \left( \frac{k}{n} \right)^{\frac{\beta}{D}}.\]
\label{thm:gen_bias_bound}
\end{theorem}

The statement for divergences is similar, assuming that $q$ is also
$\beta$-H\"older continuous with constant $L$ and strictly positive on
$\X^\circ$. Specifically, we get the same bound if we replace $M_{f,o}$ with
\[M_{f,p}(x)
  := \sup_{(w,z) \in \left[ p_*(x), p^* \right]
                     \times \left[ q_*(x), q^* \right]}
      \left| \frac{\partial}{\partial w} f(w, z) \right|\]
and define $M_{f,q}$ similarly (i.e., with $\frac{\partial}{\partial z}$) and
we assume that
\[C_f
  := \E_{X \sim p} \left[
        \frac{M_{f,p}(X)}{\left( p_*(X) \right)^{\frac{\beta}{D}}}
  \right]
  + \E_{X \sim p} \left[
        \frac{M_{f,q}(X)}{\left( q_*(X) \right)^{\frac{\beta}{D}}}
  \right]
  < \infty.\]

As an example of the applicability of Theorem \ref{thm:gen_bias_bound},
consider estimating the Shannon entropy. Then, $f(z) = \log(x)$, and so
we need $C_f = \int_\X \left( p_*(x) \right)^{-\beta/D} \, d\mu(x) < \infty$.

The assumption $C_f < \infty$ is not immediately transparent. For the
functionals in Table \ref{table:bias_corrections}, $C_f$ has the form
$\int_\X \left( p(x) \right)^{-c} \, dx$, for some $c > 0$, and hence
$C_f < \infty$ intuitively means $p(x)$ cannot approach zero too quickly as
$\dist(x, \partial\X) \to 0$. The following lemma gives a formal sufficient
condition:
\begin{lemma}
{\bf (Boundary Condition)}
Let $c > 0$. Suppose there exist $b_\partial \in (0, \frac{1}{c})$,
$c_\partial, \rho_\partial > 0$ such that, for all $x \in \X$ with
$\e(x) := \dist(x, \partial\X) < \rho_\partial$,
$p(x) \geq c_\partial \e^{b_\partial}(x)$. Then,
$\int_\X \left( p_*(x) \right)^{-c} \, d\mu(x) < \infty$.
\label{lemma:negative_integral_bound}
\end{lemma}

Now, we turn to bounding the variance.
Although the fixed-$k$ estimator is an empirical mean, because the terms being
averaged (functions of $k$-NN distances) are dependent, it is not obvious how
to go about bounding the variance of the estimator. We generalize the approach
used by \citet{biau15EntropyKNN} to prove a variance bound on the KL estimator
of Shannon entropy. The key insight is to use the geometric fact that, in
$(\R^D, \|\cdot\|_p)$, there exists a constant $N_{k,D}$ (independent of $n$)
such that any sample $X_i$ can be amongst the $k$-nearest neighbors of at most
$N_{k,D}$ other samples. Hence, at most $N_{k,D} + 1$ of the terms in
(\ref{eq:plug_in_est}) can change when a single $X_i$ is added, leading to a
variance bound via the Efron-Stein inequality \citep{efronStein81}, which
bounds the variance of a function of random variables in terms of its changes
when its arguments are resampled.

\begin{theorem}
{\bf (Variance Bound)}
Suppose that $\B \circ f$ is continuously differentiable and strictly monotone.
Assume that $C_{f,p} := \E_{X \sim P} \left[ \B^2(f(p_*(X))) \right] < \infty$,
and that $C_f := \int_0^\infty e^{-y} y^k f(y) < \infty$. Then, for
\[C_V
  := 2\left( 1 + N_{k,D} \right)
      \left( 3 + 4k \right)
      \left( C_{f,p} + C_f \right),
  \quad \mbox{ we have } \quad
  \Var \left[ \hat F_\B(P) \right]
  \leq \frac{C_V}{n}.\]
\label{thm:variance_bound}
\end{theorem}

As an example, if $f = \log$ (as in Shannon entropy), then, since $\B$ is an
additive constant, we simply require $\int_\X p(x) \log^2(p_*(x)) < \infty$.

In general, $N_{k,D}$ is of the order $k2^{cD}$, for some $c > 0$. Our bound is
likely quite loose in $k$; in practice, $\Var \left[ \hat F_\B(P) \right]$
typically decreases somewhat with $k$.

\section{Conclusions and discussion}
In this paper, we gave finite-sample bias and variance error bounds for a class
of fixed-$k$ estimators of functionals of probability density functions,
including the entropy and divergence estimators in Table
\ref{table:bias_corrections}. The bias and variance bounds in turn imply a
bound on the mean squared error (MSE) of the bias-corrected estimator via the
usual decomposition into squared bias and variance:
\begin{corollary}
{\bf (MSE Bound)}
Under the conditions of Theorems \ref{thm:gen_bias_bound} and
\ref{thm:variance_bound},
\begin{equation}
\E \left[ \left( \hat H_k(X) - H(X) \right)^2 \right]
  \leq C_f^2 L^2 \left( \frac{k}{n} \right)^{2\beta/D} + \frac{C_V}{n}.
\label{ineq:MSE_bound}
\end{equation}
\end{corollary}

{\bf Choice of $k$:} It is worth noting that, contrary to the name, fixing $k$
is not \emph{required} for ``fixed-$k$'' estimators. Indeed,
\citet{perez08estimation} empirically studied the effects of changing $k$ with
$n$, finding that fixing $k = 1$ gave the best results for estimating $F(P)$.
However, it appears there has been no formal theoretical justification for
fixing $k$ in estimation problems. Assuming the tightness of our bias bound in
$k$, we provide this in a worst-case sense: since the bias bound is
nondecreasing in $k$ and our variance bound is no larger than the minimax MSE
rate for most such estimation problems, we cannot improve the (worst-case)
convergence rate of estimators by reducing variance (i.e., by increasing $k$).
It is worth noting, however, that \citet{perez08estimation} found increasing
$k$ quickly (e.g., $k = n/2$) was \emph{best} for certain hypothesis tests
based on these estimators. Intuitively, this is because minimizing is somewhat
less important that minimizing variance problematic for testing problems. 

%
%
%

\subsubsection*{Acknowledgments}
Omitted for anonymity.

{
\fontsize{2.54mm}{0pt}\selectfont
\setlength{\bibsep}{0pt}
\bibliography{biblio}
\bibliographystyle{plainnat}
}

\appendix

\section{A More General Setting}

\setcounter{theorem}{0}

In the main paper, for the sake of clarity, we discussed only the setting of
distributions on the $D$-dimensional unit cube $[0, 1]^D$.
For sake of generality, we prove our results in the significantly more general
setting of a set equipped with a metric, a base measure, a probability density,
and an appropriate definition of dimension. This setting subsumes Euclidean
spaces, in which $k$-NN methods are usually analyzed, but also includes, for
instance, Riemannian manifolds.

\label{sec:setting}
\begin{definition}
{\bf (Metric Measure Space):}
A quadruple $(\XX, d, \Sigma, \mu)$ is called a \emph{metric measure space} if
$(\XX, d)$ is a complete metric space, $(\XX, \Sigma, \mu)$ is a
$\sigma$-finite measure space, and $\Sigma$ contains the Borel $\sigma$-algebra
induced by $d$.
\label{def:metric_meas_space}
\end{definition}

\begin{definition}
{\bf (Scaling Dimension):}
A metric measure space $(\XX, d, \Sigma, \mu)$ has \emph{scaling dimension}
$D \in [0, \infty)$ if there exist constants $\mu_*, \mu^*> 0$ such that,
$\forall r > 0$, $x \in \XX$, $\mu_* \leq \frac{\mu(B(x, r))}{r^D} \leq \mu^*$.
\footnote{$B(x, r) := \{y \in \XX : d(x, y) < r\}$ denotes the open ball of
radius $r$ centered at $x$.}
\label{def:dim}
\end{definition}

\begin{remark}
The above definition of dimension coincides with $D$ in $\R^D$, where, under the
$L^p$ metric and Lebesgue measure,
\[\mu_* = \mu^* = \frac{\left( 2\Gamma(1 + 1/p) \right)^D}{\Gamma(1 + D/p)}\]
is the usual volume of the unit ball. However, it is considerably more general
than the vector-space definition of dimension. It includes, for example, the
case that $\XX$ is a smooth Riemannian manifold, with the standard metric and
measure induced by the Riemann metric. In this case, our results scale with the
\emph{intrinsic} dimension of data, rather than the dimension of a space in
which the data are embedded. Often, $\mu_* = \mu^*$, but leaving these distinct
allows, for example, manifolds with boundary. The scaling dimension is slightly
more restrictive than the well-studied doubling dimension of a measure,
\citep{luukkainen98doublingMetricMeasure} which enforces only an upper bound on
the rate of growth.
\end{remark}

\setcounter{theorem}{1}

\section{Proofs of Lemmas}
\begin{lemma}
Consider a metric measure space $(\XX, d, \Sigma, \mu)$ of scaling dimension
$D$, and a $\mu$-absolutely continuous probability measure $P$, with density
function $p : \XX \to [0, \infty)$ supported on
\[\X := \{x \in \XX : p(x) > 0\}.\]
If $p$ is continuous on $\X$, then, for any $\rho > 0$, there exists a function
$p_* : \X \to (0, \infty)$ such that
\[0
  < p_*(x)
  \leq \inf_{r \in (0, \rho]} \frac{P(B(x, r))}{\mu(B(x, r))},
  \quad \forall x \in \X,\]
and, if $p$ is bounded above by $p^* := \sup_{x \in \X} p(x) < \infty$, then
\[\sup_{r \in (0, \rho]} \frac{P(B(x, r))}{\mu(B(x, r))}
  \leq p^* < \infty,
  \quad \forall r \in (0, \rho],\]
\end{lemma}
\begin{proof}
Let $x \in \X$. Since $p$ is continuous and strictly positive at $x$, there
exists $\e \in (0, \rho]$ such that and, for all $y \in B(x, \e)$,
$p(y) \geq p(x)/2 > 0$. Define
\[p_*(x)
  := \frac{p(x)}{2} \frac{\mu_*}{\mu^*} \left( \frac{\e}{\rho} \right)^D.\]
Then, for any $r \in (0, \rho]$, since $P$ is a non-negative measure, and $\mu$
has scaling dimension $D$,
\begin{align*}
P(B(x, r))
  \geq P(B(x, \e r/\rho))
& \geq \mu(B(x, \e r/\rho)) \min_{y \in B(x, \e r/\rho)} p(y) \\
& \geq \mu(B(x, \e r/\rho)) \frac{p(x)}{2} \\
& \geq \frac{p(x)}{2} \mu_* \left( \frac{\e r}{\rho} \right)^D
  = p_*(x) \mu^* r^D
  \geq p_*(x) \mu(B(x, r)).
\end{align*}
Also, trivially, $\forall r \in (0, \rho]$,
\begin{align*}
P(B(x, r))
  \leq \mu(B(x, r)) \max_{y \in B(x, r\rho/\e)} p(y)
  \leq p^*(x) \mu(B(x, r)).
\end{align*}
\end{proof}

\begin{lemma}
Consider a metric measure space $(\XX, d, \Sigma, \mu)$ of scaling dimension
$D$, and a $\mu$-absolutely continuous probability measure $P$, with continuous
density function $p : \XX \to [0, \infty)$ supported on
\[\X := \{x \in \XX : p(x) > 0\}.\]
For $x \in \X$, if $r > \left( \frac{k}{p_*(x) n} \right)^{1/D}$, then
\[\pr \left[ \e_k(x) > r \right]
  \leq e^{-p_*(x) r^D n} \left( e\frac{p_*(x) r^D n}{k} \right)^k.\]
and, if
$r \in \left[ 0, \left( \frac{k}{p^* n} \right)^{1/D}\right)$, then
\[\pr \left[ \e_k(x) \leq r \right]
  \leq e^{-p_*(x) r^D n} \left( \frac{e p^* r^D n}{k} \right)^{kp_*(x)/p^*}.\]
\label{lemma:KNN_concentration_appendix}
\end{lemma}
\begin{proof}
Notice that, for all $x \in \X$ and $r > 0$,
\[\sum_{i = 1}^n 1_{\{X_i \in B(x, r)\}}
  \sim \mbox{Binomial} \left( n, P(B(x, r)) \right),\]
and hence that many standard concentration inequalities apply. Since we are
interested in small $r$ (and hence small $P(B(x, r))$), we prefer bounds on
relative error, and hence apply multiplicative Chernoff bounds. If
$r > \left( k/(p_*(x) n) \right)^{1/D}$, then, by definition of $p_*$,
$P(B(x, r)) < k/n$, and so, applying the multiplicative Chernoff bound with
$\delta := \frac{p_*(x) r^D n - k}{p_*(x) r^D n} > 0$ gives
\begin{align*}
\pr \left[ \e_k(x) > r \right]
& = \pr \left[ \sum_{i = 1}^n 1_{\{X_i \in B(x, r)\}} < k \right] \\
& \leq \pr \left[ \sum_{i = 1}^n 1_{\{X_i \in B(x, r)\}}
               < (1 - \delta) n P(B(x, r)) \right] \\
& \leq \left( \frac{e^{-\delta}}
                   {(1 - \delta)^{(1 - \delta)}} \right)^{n P(B(x, r))} \\
& = e^{-p_*(x) r^D n} \left( \frac{e p_*(x) r^D n}{k} \right)^k.
\end{align*}
Similarly, if $r < \left( k/(p^* n) \right)^{1/D}$, then, applying the
multiplicative Chernoff bound with
$\delta := \frac{k - p^* r^D n}{p^* r^D n} > 0$,
\begin{align*}
\pr \left[ \e_k(x) < r \right]
& = \pr \left[ \sum_{i = 1}^n 1_{\{X_i \in B(x, r)\}} \geq k \right] \\
& \leq \pr \left[ \sum_{i = 1}^n 1_{\{X_i \in B(x, r)\}}
               \geq (1 + \delta) n P(B(x, r)) \right] \\
& \leq \left( \frac{e^\delta}{(1 + \delta)^{(1 + \delta)}} \right)^{n P(B(x, r))} \\
& \leq e^{- p_*(x) r^D n} \left( \frac{e p^* r^D n}{k} \right)^{k p_*(x)/p^*}
\end{align*}
\end{proof}

The bound we prove below is written in a somewhat different form from the
version of Lemma \ref{lemma:KNN_functional_bdd_appendix} in the main paper.
This form follows somewhat more intuitively from Lemma
\ref{lemma:KNN_concentration_appendix}, but does not make obvious the
connection to the asymptotic Erlang distribution. To derive the form in the
paper, one simply integrates the integral below by parts, plugs in the
function
$x \mapsto f \left( p_*(x) \middle/ \frac{k/n}{c_D \e_k^D(x)} \right)$,
and applies the bound $(e/k)^k \leq \frac{e}{\sqrt{k}\Gamma(k)}$.

\begin{lemma}
Consider the setting of Lemma \ref{lemma:KNN_concentration_appendix} and assume $\X$ is
compact with diameter $\rho := \sup_{x,y \in \X} d(x, y)$. Suppose
$f : (0, \rho) \to \R$ is continuously differentiable, with $f' > 0$.
Then, for any $x \in \X$, we have the upper bound
\begin{align}
\E \left[ f_+(\e_k(x)) \right]
  \leq f_+\left( \left( \frac{k}{p_*(x) n} \right)^{\frac{1}{D}} \right)
  + \frac{(e/k)^k}{D(np_*(x))^{\frac{1}{D}}}
    \int_k^{n p_*(x) \rho^D}
        \hspace{-6mm}
        e^{-y} y^{\frac{Dk + 1 - D}{D}}
        f' \left( \left( \frac{y}{n p_*(x)} \right)^{\frac{1}{D}} \right) 
    \, dy
\label{ineq:KNN_functional_upper_bdd_appendix}
\end{align}
and the lower bound
\begin{align}
\E \left[ f_-(\e_k(x)) \right]
  \leq f_-\left( \left( \frac{k}{p^* n} \right)^{\frac{1}{D}} \right)
  + \frac{\left( e/\kappa(x) \right)^{\kappa(x)}}{D\left( n p_*(x) \right)^{\frac{1}{D}}}
        \int_0^{\kappa(x)}
            \hspace{-6mm}
            e^{-y}
            y^{\frac{D\kappa(x) + 1 - D}{D}}
            f' \left( \left( \frac{y}{np_*(x)} \right)^{\frac{1}{D}} \right)
        \, dy,
\label{ineq:KNN_functional_lower_bdd_appendix}
\end{align}
where $f_+(x) = \max\{0, f(x)\}$ and $f_-(x) = -\min\{0, f(x)\}$ denote the
positive and negative parts of $f$, respectively, and $\kappa(x) := k p_*(x)/p^*$.
\label{lemma:KNN_functional_bdd_appendix}
\end{lemma}
\begin{proof}
For notational simplicity, we prove the statement for
$g(x) = f \left( n p_*(x) x^D \right)$; the main result follows by substituting
$f$ back in.

Define
\[\e_0^+ = f_+\left( \left( \frac{k}{p_*(x) n} \right)^{\frac{1}{D}} \right)
  \; \mbox{ and } \;
\e_0^- = f_-\left( \left( \frac{k}{p^* n} \right)^{\frac{1}{D}} \right).\]
Writing the expectation in terms of the survival function,
\begin{align}\
\notag
\E \left[ f_+(\e_k(x)) \right]
& = \int_0^\infty \pr \left[ f(\e_k(x)) > \e \right] \, d\e \\
\notag
& = \int_0^{\e_0^+} \pr \left[ f(\e_k(x)) > \e \right] \, d\e
\notag
  + \int_{\e_0^+}^{f_+(\rho)} \pr \left[ f(\e_k(x)) > \e \right] \, d\e, \\
\label{ineq:KNN_stat_upper_piece2}
& \leq \e_0^+
  + \int_{\e_0^+}^{f_+(\rho)} \pr \left[ f(\e_k(x)) > \e \right] \, d\e,
\end{align}
since $f$ is non-decreasing and $\pr \left[ \e_k(x) > \rho \right] = 0$.
By construction of $\e_0^+$, for all $\e > \e_0^+$,
$f\inv(\e) > \left( k/(p_*(x) n) \right)^{1/D}$. Hence, applying Lemma
\ref{lemma:KNN_concentration_appendix} followed by the change of variables
$y = np_*(x) \left( f\inv(\e) \right)^D$ gives
\footnote{$f$ need not be surjective, but the generalized inverse
$f\inv : [-\infty, \infty] \to [0, \infty]$ defined by
$f\inv(\e) := \inf \{x \in (0, \infty) : f(x) \geq \e\}$ suffices here.}
\begin{align*}
\int_{\e_0^+}^{f_+(\rho)} \pr \left[ \e_k(x) > f\inv(\e) \right] \, d\e
& \leq \int_{\e_0^+}^{f_+(\rho)}
           e^{-n p_*(x) \left( f\inv(\e) \right)^D}
           \left( \frac{enp_*(x) \left( f\inv(\e) \right)^D}{k} \right)^k
       \, d\e \\
& = \frac{(e/k)^k}{D(np_*(x))^{\frac{1}{D}}}
      \int_k^{n p_*(x) \rho^D}
          e^{-y} y^{\frac{kD + 1 - D}{D}}
          f' \left( \left( \frac{y}{np_*(x)} \right)^{\frac{1}{D}} \right) 
      \, dy,
\end{align*}
Together with (\ref{ineq:KNN_stat_upper_piece2}), this gives the
upper bound (\ref{ineq:KNN_functional_upper_bdd_appendix}). Similar steps give
\begin{equation}
\E \left[ f(\e_k(x)) \right]
  \leq \e_0^-
  + \int_{\e_0^-}^{f_-(0)} \pr \left[ f(\e_k(x)) < -\e \right] \, d\e.
\label{ineq:KNN_stat_lower_part}
\end{equation}
Applying Lemma \ref{lemma:KNN_concentration_appendix} followed the change of variables
$y = n p_*(x) \left( f\inv(-\e) \right)^D$ gives
\begin{align*}
\int_{\e_0^-}^{f_-(\rho)} \pr \left[ \e_k(x) < f\inv(-\e) \right] \, d\e
& \leq  \frac{\left( e/\kappa(x) \right)^{\kappa(x)}}{D\left( n p_*(x) \right)^{\frac{1}{D}}}
        \int_0^{\kappa(x)}
            e^{-y}
            y^{\frac{D\kappa(x) + 1 - D}{D}}
            f' \left( \left( \frac{y}{np_*(x)} \right)^{\frac{1}{D}} \right)
        \, dy
\end{align*}
Together with inequality (\ref{ineq:KNN_stat_lower_part}), this gives the
result (\ref{ineq:KNN_functional_lower_bdd_appendix}).
\end{proof}

\subsection{Applications of Lemma \ref{lemma:KNN_functional_bdd_appendix}}
When $f(x) = \log(x)$, (\ref{ineq:KNN_functional_upper_bdd_appendix}) gives
\begin{align}
\notag
\E \left[ \log_+(\e_k(x)) \right]
  \leq \frac{1}{D} \log_+ \left( \frac{k}{p_*(x) n} \right)
  + \left( \frac{e}{k} \right)^k \frac{\Gamma(k, k)}{D}
  \leq \frac{1}{D}
       \left( \log_+ \left( \frac{k}{p_*(x) n} \right) + 1 \right)
\label{ineq:pos_log_stat_appendix}
\end{align}
and (\ref{ineq:KNN_functional_lower_bdd_appendix}) gives
\footnote{$\Gamma(s, x) := \int_x^\infty t^{s - 1} e^{-t} \, dt$ and
$\gamma(s, x) := \int_0^x t^{s - 1} e^{-t} \, dt$ denote the upper and lower
incomplete Gamma functions respectively. We used the bounds
$\Gamma(s, x), x\gamma(s, x) \leq x^s e^{-x}$.}
\begin{align}
\E \left[ \log_-(\e_k(x)) \right]
& \leq \frac{1}{D}
        \left( \log_-\left( \frac{k}{p^* n} \right)
              + \left( \frac{e}{\kappa(x)} \right)^{\kappa(x)} \gamma(\kappa(x), \kappa(x))
        \right) \\
& \leq \frac{1}{D}
        \left( \log_-\left( \frac{k}{p^* n} \right) + \frac{1}{\kappa(x)} \right).
\label{ineq:neg_log_stat_appendix}
\end{align}
For $\alpha > 0$, $f(x) = x^\alpha$,
(\ref{ineq:KNN_functional_upper_bdd_appendix})
gives
\begin{align}
\notag
\E \left[ \e_k^\alpha(x) \right]
& \leq \left( \frac{k}{p_*(x) n} \right)^{\frac{\alpha}{D}}
    + \left( \frac{e}{k} \right)^k
      \frac{\alpha\Gamma\left( k + \alpha/D, k \right)}
           {D(np_*(x))^{\alpha/D}} \\
& \leq C_2
       \left( \frac{k}{p_*(x) n} \right)^{\frac{\alpha}{D}},
\label{ineq:pos_moment_stat_appendix}
\end{align} 
where $C_2 = 1 + \frac{\alpha}{D}$. For any
$\alpha \in [-D\kappa(x), 0]$, when $f(x) = -x^\alpha$,
(\ref{ineq:KNN_functional_lower_bdd_appendix}) gives
\begin{align}
\E \left[ \e_k^\alpha(x) \right]
& \leq \left( \frac{k}{p^* n} \right)^{\frac{\alpha}{D}}
    + \left( \frac{e}{\kappa(x)} \right)^{\kappa(x)}
      \frac{\alpha\gamma\left( \kappa(x) + \alpha/D, \kappa(x) \right)}
           {D(np_*(x))^{\alpha/D}} \\
& \leq C_3
    \left( \frac{k}{p^* n} \right)^{\frac{\alpha}{D}},
\label{ineq:neg_moment_stat_appendix}
\end{align}
where
$C_3 = 1 + \frac{\alpha}{D\kappa(x) + \alpha}$.

%
\section{Proof of Bias Bound}

\begin{theorem}
Consider the setting of Lemma \ref{lemma:KNN_concentration_appendix}. Suppose
Suppose $p$ is $\beta$-H\"older continuous, for some $\beta \in (0, 2]$.
Let $f : (0, \infty) \to \R$ be differentiable, and define
$M_f : \X \to [0, \infty)$ by
\[M_f(x)
  := \sup_{z \in \left[ \frac{p_*(x)}{\mu^*}, \frac{p^*}{\mu_*} \right]}
      \left\| \nabla f(z) \right\|\]
(assuming this quantity is finite for almost all $x \in \X$). Suppose that
\[C_M
  := \E_{X \sim p} \left[
        \frac{M_f(X)}{\left( p_*(X) \right)^{\frac{\beta}{D}}}
  \right]
  < \infty.\]
Then, for $C_B := C_M L$,
\[\left| \E_{X, X_1,\dots,X_n \sim P} \left[ f(p_{\e_k(X)}(X)) \right]
          - F(p) \right|
  \leq C_B \left( \frac{k}{n} \right)^{\frac{\beta}{D}}.\]
\label{thm:gen_bias_bound_appendix}
\end{theorem}
\begin{proof}
By construction of $p_*$ and $p^*$,
\[p_*(x) \leq p_\e(x) = \frac{P(B(x, \e))}{\mu(B(x, \e))} \leq p^*.\]
Also, by the Lebesgue differentiation theorem
\citep{lebesgue10differentiation}, for $\mu$-almost all $x \in \X$,
\[p_*(x) \leq p(x) \leq p_*.\]
For all $x \in \X$, applying the mean value theorem followed by inequality
(\ref{ineq:pos_moment_stat_appendix}),
\begin{align*}
\E_{X_1,\dots,X_n \sim p} \left[ \left| f(p(x)) - f(p_{\e_k(x)}(x)) \right| \right]
& \leq \E_{X_1,\dots,X_n \sim p} \left[ \left\| \nabla f(\xi(x)) \right\| \left| p(x) - p_{\e_k(x)}(x) \right| \right] \\
& \leq M_f(x) \E_{X_1,\dots,X_n \sim p} \left[ \left| p(x) - p_{\e_k(x)}(x) \right| \right] \\
& \leq \frac{M_f(x) LD}{D + \beta}
       \E_{X_1,\dots,X_n \sim P} \left[ \e_k^{\beta}(x) \right] \\
& \leq \frac{C_2 M_f(x) LD}{D + \beta} \left( \frac{k}{p_*(x) n}
\right)^{\frac{\beta}{D}}
\end{align*}
Hence,
\begin{align*}
\left| \E_{X_1,\dots,X_n \sim p} \left[ F(p) - \hat F(p) \right] \right|
& = \left|
      \E_{X \sim p} \left[
          \E_{X_1,\dots,X_n \sim p} \left[ f(p(X)) - f(p_{\e_k(X)}(X)) \right]
      \right]
  \right| \\
& \leq \frac{C_2 LD}{D + \beta}
    \E_{X \sim p} \left[
        \frac{M_f(X)}{\left( p_*(X) \right)^{\frac{\beta}{D}}}
    \right]
    \left( \frac{k}{n} \right)^{\frac{\beta}{D}}
  = \frac{C_2 C_M LD}{D + \beta}
    \left( \frac{k}{n} \right)^{\frac{\beta}{D}}.
\end{align*}
\end{proof}

\begin{lemma}
Let $c > 0$. Suppose there exist $b_\partial \in (0, \frac{1}{c})$,
$c_\partial, \rho_\partial > 0$ such that for all $x \in \X$ with
$\e(x) := \dist(x, \partial\X) < \rho_\partial$,
$p(x) \geq c_\partial \e^{b_\partial}(x)$. Then,
\[\int_\X \left( p_*(x) \right)^{-c} \, d\mu(x) < \infty.\]
\label{lemma:negative_integral_bound_appendix}
\end{lemma}
\begin{proof}
Let $\X_\partial := \{x \in \X : \dist(x, \partial\X) < \rho_\partial\}$ denote
the region within $\rho_\partial$ of $\partial\X$.
Since $p_*$ is continuous and strictly positive on the compact set
$\X\sminus\X_\partial$, it has a positive lower bound
$\ell := \inf_{x \in \X\sminus\X_\partial}$ on this set, and it suffices to
show
\[\int_{\X\sminus\X_\partial} \left( p_*(x) \right)^{-c} \, d\mu(x) < \infty.\]
For all $x \in \X_\partial$,
\[p_*(x)
  \geq \frac{\min\{\ell, c_\partial \e^{b_\partial}(x)\}}{\mu(B(x, \sqrt{D}))}
.\]
Hence, 
\[\int_{\X\sminus\X_\partial} \left( p_*(x) \right)^{-c} \, d\mu(x)
  \leq \int_{\X\sminus\X_\partial} \ell^{-c} \, d\mu(x)
    + \int_{\X\sminus\X_\partial} c_\partial^{-c} \e^{-b_\partial/c}(x) \, d\mu(x).\]
The first integral is trivially bounded by $\ell^{-c}$. Since $\partial\X$
is the union of $2D$ ``squares'' of dimension $D - 1$, the second integral can
be reduced to the sum of $2D$ integrals of dimension $1$, giving the bound
\[2D c_\partial^{-c} \int_0^{\rho_\partial} x^{-b_\partial/c}(x) \, dx.\]
Since $b_\partial/c < 1$, the integral is finite.
\end{proof}

\section{Proof of Variance Bound}

\begin{theorem}
{\bf (Variance Bound)}
Suppose that $\B \circ f$ is continuously differentiable and strictly monotone.
Assume that $C_{f,p} := \E_{X \sim P} \left[ \B^2(f(p_*(X))) \right] < \infty$,
and that $C_f := \int_0^\infty e^{-y} y^k f(y) < \infty$. Then, for
\[C_V
  := 2\left( 1 + N_{k,D} \right)
      \left( 3 + 4k \right)
      \left( C_{f,p} + C_f \right),
  \quad \mbox{ we have } \quad
  \Var \left[ \hat F_\B(P) \right]
  \leq \frac{C_V}{n}.\]
\label{thm:variance_bound_appendix}
\end{theorem}
\begin{proof}
For convenience, define
\[H_i :=
  \B \left(
    f \left(
      \frac{k/n}{\mu\left( B(X_i, \e_k(X_i)) \right)}
    \right)
  \right).\]
By the Efron-Stein inequality \citep{efronStein81} and the fact that the
$\hat F_\B(P)$ is symmetric in $X_1,\dots,X_n$,
\begin{align*}
\Var \left[ \hat F_\B(P) \right]
& \leq \frac{n}{2}
       \E \left[ \left( \hat F_\B(P) - F_\B'(P) \right)^2 \right] \\
& \leq n \E \left[ \left( \hat F_\B(P) - F_{2:n} \right)^2
                 + \left( \hat F_\B'(P) - F_{2:n} \right)^2 \right] \\
& = 2n \E \left[ \left( \hat F_\B(P) - F_{2:n} \right)^2 \right],
\end{align*}
where $\hat F_{\B}'(P)$ denotes the estimator after $X_1$ is resampled, and
$F_{2:n} := \frac{1}{n} \sum_{i = 2}^n H_i$. Then,
\begin{align*}
n(\hat F_n(P) - F_{2:n})
& = H_1
  + \sum_{i = 2}^n 1_{E_i} \left( H_i - H_i' \right),
\end{align*}
where $1_{E_i}$ is the indicator function of the event
$E_i = \{\e_k(X_i) \neq \e_k'(X_i)\}$. By Cauchy-Schwarz followed by the
definition of $N_{k,D}$,
\begin{align*}
n^2(\hat F_n(P) - \hat F_{n - 1}(P))^2
& = \left( 1 + \sum_{i = 2}^n 1_{E_i} \right)
  \left( H_1^2
  + \sum_{i = 2}^n 1_{E_i} \left( H_i - H_i' \right)^2 \right) \\
& = \left( 1 + N_{k,D} \right)
  \left( H_1^2
  + \sum_{i = 2}^n 1_{E_i} \left( H_i - H_i' \right)^2 \right) \\
& \leq \left( 1 + N_{k,D} \right)
  \left( H_1^2
  + 2\sum_{i = 2}^n 1_{E_i} \left( H_i^2 + H_i'^2 \right) \right).
\end{align*}
Taking expectations, since the terms in the summation are identically
distributed, we need to bound
\begin{align}
\label{exp:variance_term_1}
& \E \left[ H_1^2 \right], \\
\label{exp:variance_term_2}
& (n - 1) \E \left[ 1_{E_2} H_2^2 \right] \\
\label{exp:variance_term_3}
\mbox{ and } \quad
& (n - 1) \E \left[ 1_{E_2} H_2'^2 \right].
\end{align}
{\bf Bounding \eqref{exp:variance_term_1}:} Note that
\[\E \left[ H_1^2 \right]
  = \E \left[ \B^2 \left( f \left( \hat p_k(X_1) \right) \right) \right]
  = \E \left[ \B^2 \left( g \left( \frac{p_*(x)}{\hat p_k(x)} \right) \right) \right]\]
for $g(y) = f \left( p_*(x)/y \right)$. Applying the upper bound in Lemma
\ref{lemma:KNN_functional_bdd_appendix}, if $\B^2 \circ g$ is increasing,
\[\E \left[ H_1^2 \right]
  \leq \B^2(g(1)) + \frac{e\sqrt{k}}{\Gamma(k + 1)} C_\uparrow
  = \B^2(f(p_*(x))) + \frac{e\sqrt{k}}{\Gamma(k + 1)} C_\uparrow
.\]
If $\B^2 \circ g$ is decreasing, we instead use the lower bound in Lemma
\ref{lemma:KNN_functional_bdd_appendix}, giving a similar result. If $\B^2 \circ g$ is
not monotone (i.e., if $\B \circ g$ takes both negative and positive values),
then, since $\B \circ f$ \emph{is} monotone (by assumption), we can apply the
above steps to $\left( \B \circ g \right)_-$ and $\left( \B \circ g \right)_+$,
which \emph{are} monotone, and add the resulting bounds.

{\bf Bounding \eqref{exp:variance_term_2}:} Since
$\{\e_k(X_2) \neq \e_k'(X_2)\}$ is precisely the event that $X_1$ is
amongst the $k$-NN of $X_2$,
$\pr \left[ \e_k(X_i) \neq \e_k'(X_i) \right] = k/(n - 1)$. Thus,
since $E_2$ is independent of $\e_k(X_2)$ and
\begin{align*}
(n - 1) \E \left[ 1_{E_2} H_2^2 \right]
  = (n - 1) \E \left[ 1_{E_2} \right] \E \left[ H_2^2 \right]
  = k \E \left[ H_2^2 \right]
  = k\E \left[ H_1^2 \right],
\end{align*}
and we can use the bound for \eqref{exp:variance_term_1}.

{\bf Bounding \eqref{exp:variance_term_3}:}
Since $E_2$ is independent of $\e_{k + 1}(X_2)$ and
\begin{align*}
(n - 1) \E \left[ 1_{E_2} H_2'^2 \right]
& = (n - 1) \E \left[ 1_{E_2} \B^2 \left( f \left( \hat p_{k + 1}(X_2) \right) \right) \right] \\
& = (n - 1) \E \left[ 1_{E_2} \right] \E \left[ \B^2 \left( f \left( \hat p_{k + 1}(X_2) \right) \right) \right]
  = k \E \left[ \B^2 \left( f \left( \hat p_{k + 1}(X_2) \right) \right) \right].
\end{align*}
Hence, we can again use the same bound as for \eqref{exp:variance_term_1},
except with $k + 1$ instead of $k$.

Combining these three terms gives the final result.

\end{proof}

\end{document}